\documentclass[12pt,reqno]{amsart}

\usepackage{setspace}
\usepackage{enumitem}
\usepackage{tikz-cd}
\usepackage{tikz}
\usetikzlibrary{positioning} 

\usepackage[linktocpage=true]{hyperref}

\makeatletter
\@namedef{subjclassname@2020}{%
  \textup{2020} Mathematics Subject Classification}
\makeatother
\newtheorem*{theorem*}{Theorem A}
\newtheorem*{theorem**}{Theorem B}

\usepackage[T1]{fontenc}
\usepackage{ae,aecompl}
\usepackage{graphics,cancel,graphicx}
\usepackage{epsfig,psfrag,manfnt}
\usepackage{mathrsfs}
\usepackage{graphicx,mathabx}
\usepackage{bbm,subfigure,rotating,bm,float,mathdots,wasysym,scalerel}

\newlength{\wdth}

\usepackage{stmaryrd}

\usepackage[all,cmtip]{xy}
\usepackage{amssymb,amsmath,amsfonts,amsthm,color}

\usepackage{scalerel,stackengine}
\stackMath
\newcommand\reallywidehat[1]{%
\savestack{\tmpbox}{\stretchto{%
  \scaleto{%
    \scalerel*[\widthof{\ensuremath{#1}}]{\kern-.6pt\bigwedge\kern-.6pt}%
    {\rule[-\textheight/2]{1ex}{\textheight}}
  }{\textheight}%
}{0.5ex}}%
\stackon[1pt]{#1}{\tmpbox}%
}

\makeatletter
\newcommand{\oamp}{
\!\mathbin{\mathpalette\make@circled{\raisebox{0.1em}{\scaleobj{0.51}{\&}}}}\!}
\newcommand{\make@circled}[2]{%
  \ooalign{$\m@th#1\smallbigcirc{#1}$\cr\hidewidth$\m@th#1#2$\hidewidth\cr}%
}
\newcommand{\smallbigcirc}[1]{%
  \vcenter{\hbox{\scalebox{0.77778}{$\m@th#1\bigcirc$}}}%
}
\makeatother

\newcommand{\calD}{\mathcal{D}}

\newcommand{\calL}{\mathcal{L}}

\newcommand{\calS}{\mathcal{S}}

\newcommand{\mC}{\mathbb{C}}

\newcommand{\mN}{\mathbb{N}}

\newcommand{\mR}{\mathbb{R}}

\newtheorem{theorem}{Theorem}[section]

\newtheorem{proposition}[theorem]{Proposition}

\theoremstyle{definition}

\newtheorem{remark}[theorem]{Remark}

\theoremstyle{definition}

\theoremstyle{definition}

\theoremstyle{definition}

\begin{document}



\title[]{On the concatenability of solutions of partial differential equations}

\author[]{Sara Maad  Sasane}
\address{Centre for Mathematical Sciences \\ Lund University \\
    Box 118\\ \phantom{22100Lund} 221~00~Lund \\ Sweden}
\email{Sara.Maad$\_$Sasane@math.LTH.se}

\author[]{Amol Sasane}
\address{Department of Mathematics \\London School of Economics\\
     Houghton Street\\ London WC2A 2AE\\ United Kingdom}
\email{A.J.Sasane@LSE.ac.uk}
 
\maketitle

\vspace{-1cm}
   
\begin{abstract} 
Let $\calD'(\mR^d)$ denote the space of  distributions on $\mR^d$. 
For a  linear partial different equation $p(\frac{\partial\hfill}{\partial x_1},\cdots, \frac{\partial\hfill}{\partial x_d}, \frac{\partial\hfill}{\partial t}) u=0$ (briefly $D_pu\!=\!0$) corresponding to a polynomial $p\in \mC[\xi_1,\cdots\!, \xi_d,\tau]$, let $S_p\!:=\!\{u\in C(\mR, \calD'(\mR^d)):D_pu\!=\!0\}$. The set $S_p$ has the `concatenability property' if whenever $u_1,u_2\in S_p\cap C^1(\mR, \calD'(\mR^d))$ are such that $u_1(0)\!=\!u_2(0)$, their concatenation $u_1\oamp u_2$ (defined to be $u_1(t)$ for $t\!\le\! 0$, and $u_2(t)$ for $t\!\ge\! 0$) belongs to $S_p$. It is shown that for $p\!=\!a_0+a_1\tau+\cdots+a_{d}\tau^{d}\in \mC[\xi_1,\cdots, \xi_d][\tau]$, where $a_0,\cdots, a_{d}\in  \mC[\xi_1,\cdots, \xi_d]$ and $d\in \mN$, $S_p$ has the concatenation property if and only if $d\!=\!1$. 
\end{abstract}

\section{Introduction}

\noindent The aim of this article is to show that the set of solutions of a single scalar homogeneous linear partial different equation with constant coefficients has the concatenability property (roughly whenever $u_1, u_2$ are solutions that match at $0$, the trajectory $u_1\oamp u_2$ whose past matches with $u_1$ and future with $u_2$ is also a solution) 
if and only if the polynomial describing the differential equation is first order in $\frac{\partial\hfill}{\partial t}$. 

The motivation for considering this problem is two-fold. Firstly, it arises from the discussion of the problem of understanding the `state' and `state construction' for a control theoretic model described by a system of linear partial differential equations with constant coefficients given in \cite{Wil}. While `time' does not play a distinguished role among the independent variables at hand in \cite{Wil}, it does so in the present article\footnote{The second author had posed the question considered in this article in a private conversation with Professor Jan Willems when he was a PhD student at the Bernoulli Institute, University of Groningen in 2000.}, and generalises to partial differential equations (PDE) the ordinary differential equation (ODE) case, viewing models as `evolution' equations. In the ordinary differential equation context, the state construction problem was considered \cite{RapWil} (where even {\em system} of ODEs were considered, while we restrict ourselves just to the {\em scalar} case in the PDE case as a first simple step). See also, for e.g., \cite{RocWil},  \cite{vanRap}, \cite{Zer} for various other avatars of this problem. We also remark that within  multidimensional behavioural theory of Willems, the notions where the independent variable of time plays a distinguished role (for instance in the concepts of controllability and autonomy) were introduced by the second author (see, for example, \cite{Sas0}, \cite{Sas00}, \cite{CotSas}, \cite{SasThoWil}), and there has been revived interest in such time-relevant notions (see, for instance, \cite{NapRapRoc}, \cite{ObeSch}). 
In the ODE context, the `concatenability property' was referred to as the `Markovian property', but to avoid confusion with the terminology in \cite{Wil} in the PDE context, we prefer to use the name `concatenability property' in both the PDE and ODE contexts. Besides the motivation arising from control theory, another motivation is that the problem we consider also fits within the general framework of algebraic-analysis of partial differential equations where one tries to characterise {\em analytic} properties of the solution space $S_p=\{u:p(\frac{\partial\hfill}{\partial x_1},\cdots, \frac{\partial\hfill}{\partial x_d}, \frac{\partial\hfill}{\partial t}) u=0\}$ of a PDE  {\em algebraically} via properties of the describing polynomial $p\in \!\mC[\xi_1,\cdots\!, \xi_d,\tau]$, for instance, the classically studied `division problems' (see, e.g., \cite{Hor0}, \cite{Loj} and also \cite{SasWag}), the problem of existence of null solutions (see, e.g., \cite{Hor}, \cite{Sas}), etc.

The organisation of this short note is as follows. In Section~\ref{section_1}, we present Theorem~\ref{one_variable_thm}, which treats the concatenability property in the case of an ordinary differential equation.  To treat the partial differential equation case, we need some preliminaries on vector-valued functions and distributions, which are given in Section~\ref{section_2}. Finally, in Section~\ref{section_3}, we will show Theorem~\ref{PDE_thm}, which handles the concatenability property for partial differential equations.

\section{The ordinary differential equation case}
\label{section_1}

 \noindent Let $\calD'(\mR)$ denote the space of  distributions on $\mR$, and $\calD(\mR)$ the space of compactly supported complex-valued infinitely many times differentiable functions on $\mR$. The action of $T\in \calD'(\mR)$ on a test function $\varphi\in \calD(\mR)$ will be denoted by $\langle T, \varphi\rangle$.  Let $\delta\in \calD'(\mR)$ denote the Dirac distribution $\langle \delta ,\varphi \rangle=\varphi(0)$. The successive distributional derivatives of $\delta$ are denoted by $\delta', \delta'', \delta'''$ and by $\delta^{(k)}$ for $k\ge 4$.
 
 Consider a single scalar linear, constant coefficient, homogeneous ordinary differential equation. 
The corresponding differential operator  is then described by a polynomial $p\in \mC[t]$, with   the indeterminate $t$ replacing the derivative map $\frac{d\hfill}{dt}:\calD'(\mR)\to \calD'(\mR)$. Let $S_p$ denote the set of continuous solutions, that is, 
$$
\textstyle
S_p=\{ u\in C(\mR): p(\frac{d\hfill}{dt})u=0\text{ in the sense of distributions}\}.
$$
If $u_1,u_2\in S_p\cap C^1(\mR)$ and $u_1(0)\!=\!u_2(0)$, then we define their {\em concatenation} 
$u_1\oamp  u_2 \in C(\mR)$ by 
$$
\textstyle 
(u_1\oamp u_2)(t)=\Big\{\! \begin{array}{l} 
u_1(t)  \;\text{ if }\; t\!\le\! 0,\\
u_2(t) \;\text{ if } \; t\! \ge\!  0.
\end{array}
$$
We say that $S_p$ has the {\em concatenability property} if whenever $u_1,u_2\in S_p\cap C^1(\mR)$, their concatenation $u_1\oamp u_2$ belongs to $S_p$. 

\begin{theorem}
\label{one_variable_thm}
Let $p\in \mC[t]$ be a nonconstant polynomial. 

\noindent Then $S_p$ has the concatenability property if and only if $\deg p=1$.
\end{theorem}
\begin{proof} (If:) Let $p\!=\!a_0+a_1t$, where $a_0,a_1\in \mC$ and $a_1\!\neq\! 0$. 
The  distributional solution to  $a_0u+a_1\frac{du}{dt}\!=\!0$  
is 
$u\!=\!c e^{\lambda t}$, where $c\in \mC$ is arbitrary, and $\lambda\!:=\!-\frac{a_0}{a_1}$. If $u_1,u_2\in S_p$, then  $u_1\!=\! c_1e^{\lambda t}$ and $u_2\!=\!c_2e^{\lambda t}$ for some $c_1,c_2\in \mC$. If, moreover, $u_1(0)\!=\!u_2(0)$, then $c_1\!=\!c_2$, and so $u_1\!=\!u_2$. Thus $u_1\oamp u_2\!=\!u_1\!=\!u_2 \in S_p$. Hence $S_p$ has the concatenability property. 

\vspace{0.21cm}

\noindent (Only if:) Suppose $S_p$ has the concatenability property. 
Let $\deg p\!>\!1$, and 
$p\!=\!a_0+a_1t+\cdots +a_n t^n$, where $a_0,\cdots, a_n\in \mC$, $n\in \mN$, and $a_n\!\neq\! 0$. 
 
 \vspace{0.1cm}
 
\noindent ${\bf 1^\circ}$ $p$ has a repeated root $\lambda\in \mC$, i.e., $p\!=\!(t-\lambda)^2q$ for some $q\in \mC[t]$. Then $u\!=\!(\alpha +\beta t)e^{\lambda t} \in S_p$ for all $\alpha, \beta \in \mC$. In particular, $u_1\!:=\!e^{\lambda t}\in S_p$ and $u_2\!:=\!(1+t)e^{\lambda t}\in S_p$. Moreover, $u_1(0)\!=\!1\!=\!u_2(0)$. 
Define the locally integrable functions $f_k$, for $k\in \{0,1,2,\cdots, n\}$, by 
\begin{eqnarray*}
\textstyle
 f_0(t)\!\!\!&=&\!\!\!\Big\{ \begin{array}{ll} 
 e^{\lambda t}& \text{if }\; t\! <\! 0,\\
(1+t)e^{\lambda t} & \text{if }\; t\! >\! 0,
\end{array}
\\
f_1(t)\!\!\!&=&\!\!\!\Big\{ \begin{array}{ll} 
 \lambda e^{\lambda t}& \text{if }\; t\! <\! 0,\\
(\lambda + 1+\lambda t)e^{\lambda t} & \text{if }\; t\! >\! 0,
\end{array}
\end{eqnarray*}
and for $k\ge 2$, 
\begin{eqnarray*}
f_k(t)\!\!\!&=&\!\!\!\Big\{ \begin{array}{ll} 
\lambda^k e^{\lambda t}& \text{if }\; t\! <\! 0,\\
\lambda^{k-1}(\lambda+k+\lambda t)e^{\lambda t} & \text{if }\; t\! >\! 0.
\end{array}
\end{eqnarray*}
Using induction on $k$, we get 
 $$
\textstyle 
(\frac{d\hfill}{dt})^k (u_1\oamp u_2)\!=\!f_k +\delta^{(k-2)}+2\lambda\delta^{(k-3)}+3\lambda^2 \delta^{(k-4)}+\cdots +(k-1) \lambda^{k-2} \delta.
$$
(with $\delta^{(m)}$ for $m<0$ defined to be the zero distribution, and $f_k$ is the regular distribution given by $\langle f_k,\varphi\rangle=\int_{\mR} f_k(t)\varphi(t) dt$ for all $\varphi\in \calD(\mR)$). 
As $u_1, u_2\in S_p$, we get that the restriction of the distribution $p(\frac{d\hfill}{dt})f_k$ to the open set $\mR\setminus \{0\}$ is $0$, i.e., $(p(\frac{d\hfill}{dt})f_k)|_{\mR\setminus \{0\}}\!=\!0$. Since $p(\frac{d\hfill}{dt})(u_1\oamp u_2)\!=\!0$, 
$$
\textstyle 
{\scaleobj{0.81}{\sum\limits_{k=0}^n}} a_k( \delta^{(k-2)}+2\lambda\delta^{(k-3)}+3\lambda^2 \delta^{(k-4)}+\cdots +(k-1) \lambda^{k-2} \delta)=0.
$$
By the independence of $\delta, \delta',\delta'', \cdots$ in $\calD'(\mR)$, it follows that each $a_k\!=\!0$ for all $k\in\{2,\cdots, n\}$, i.e., $p\!=\!a_0+a_1t$, a contradiction to $\deg p>1$.

\vspace{0.1cm}

\noindent ${\bf 2^\circ}$ Suppose $p$ has no repeated roots. Let $p$ have roots at $\lambda,\mu\in \mC$, with $\lambda\neq \mu$, i.e., $p\!=\!(t-\lambda)(t-\mu)q$ for some $q\in \mC[t]$. Then  $u_1\!:=\!e^{\lambda t}\in S_p$ and $u_2\!:=\!e^{\mu t}\in S_p$ and $u_1(0)\!=\!1\!=\!u_2(0)$. 
Define the locally integrable functions $f_k$ for $k\!=\!0,1,2,\cdots, n$ by 
$$
\textstyle 
f_k(t)=\bigg\{ \begin{array}{ll} 
\lambda^k e^{\lambda t}& \text{ if } \;t\!<\!0,\\
\mu^k e^{\mu t} & \text{ if }\; t\!>\!0.
\end{array}
$$
Then 
 $
\textstyle 
(\frac{d\hfill}{dt})^k (u_1\oamp\;\! u_2)\!=\!f_k + (\mu-\lambda)\delta^{(k-2)}+ \cdots+(\mu^{k-2}-\lambda^{k-2})\delta.
$ 
Since $u_1, u_2\in S_p$, we get $(p(\frac{d\hfill}{dt})f_k)|_{\mR\setminus \{0\}}\!=\!0$. As $p(\frac{d\hfill}{dt})(u_1\oamp u_2)\!=\!0$, we have
$$
\textstyle 
{\scaleobj{0.81}{\sum\limits_{k=0}^n}} a_k((\mu-\lambda)\delta^{(k-2)}+ \cdots+(\mu^{k-2}-\lambda^{k-2})\delta)=0.
$$
By the independence of $\delta, \delta',\delta'', \cdots$ in $\calD'(\mR)$, it follows that each $a_k\!=\!0$, $k\!\in \{2,\cdots, n\}$, i.e., $p\!=\!a_0+a_1t$, a contradiction to $\deg p>1$. 
\end{proof}

\section{Vector-valued functions and distributions}
\label{section_2}

\noindent For preliminaries on  vector-valued distributions, we refer the reader to \cite[Chapter~1, \S11]{Car}.

Let $\calD'(\mR^d)$ denote the space of all complex-valued distributions on $\mR^d$. The space of  test functions on $\mR^d$ is denoted by $\calD(\mR^d)$ (that is, $\calD(\mR^d)$ is the set of all compactly supported complex-valued $C^\infty$ functions on $\mR^d$). For a  distribution $T\in \calD'(\mR^d)$ and a test function $\psi\in \calD(\mR^d)$, the action of $T$ on $\psi$ will be denoted by $\langle T,\psi\rangle$. 

We now consider  maps $t\mapsto u(t)$, where $u(t)\in\calD'(\mR^d)$, whose differential calculus (see the `Jump Rule', Proposition~\ref{prop_jump} below)  will play an important role in the sequel. 
Let $I\subset \mR$ be an open set. 
A function $u: I\to \calD'(\mR^d)$ is {\em locally integrable} (respectively {\em continuous}, respectively $C^k$ for some $k\in \mN$) if for all $\psi \in \calD(\mR^d)$, the function $t\mapsto \langle u(t),\psi\rangle: I\to \mC$ is locally integrable (respectively {\em continuous}, respectively $C^k$), and we write $u\in \text{L}^1_{\text{loc}}(I, \calD'(\mR^d))$ (respectively $C(I, \calD'(\mR^d))$, respectively 
$C^k(I, \calD'(\mR^d))$). If $u\in C^k$ for all $k\in \mN$, we write $u\in C^\infty(I, \calD'(\mR^d))$.
If $u\in C^1(I, \calD'(\mR^d))$, then we define its derivative $\frac{du}{dt}$ by
$$
\textstyle 
\langle \frac{du}{dt}(t), \psi\rangle =
\frac{d\hfill}{dt} \langle u(t), \psi\rangle 
=
\lim\limits_{h\to 0} \langle \frac{u(t+h)-u(t)}{h}, \psi\rangle \text{ for all }\psi \in \calD(\mR^d).
$$
Taking $h$ as a sequence, for example $\frac{1}{n}$, $n\in \mN$, we see that $\frac{du}{dt}(t)$ is the pointwise limit of a sequence of distributions in $\calD'(\mR^d)$, from which it follows using the Banach-Steinhaus theorem in the setting of the barelled space $\calD(\mR^d)$, that $\frac{du}{dt}(t)\in \calD'(\mR^d)$.

The space of continuous linear maps from $\calD(\mR)$ to 
$\calD'(\mR^d)$ will be denoted by $\calL(\calD(\mR),\calD'(\mR^d))$. 
For a locally integrable $u\in \text{L}^1_{\text{loc}}(I, \calD'(\mR^d))$,  we define $T_{u} \in \calL(\calD(\mR),\calD'(\mR^d))$, called {\em the regular distribution corresponding to $u$}, by 
$$
\textstyle 
\langle \langle T_{u}, \varphi \rangle , \psi\rangle  
:=
\int_\mR \varphi(t) \langle u(t), \psi\rangle \;\!dt 
\;\text{ for all }\varphi \in \calD(\mR) \text{ and } \psi \in\calD(\mR^d).
$$ 
If $\sigma\in \calD'(\mR^d)$, then we define $\delta\otimes \sigma \in \calL(\calD(\mR),\calD'(\mR^d))$ by 
$$
\textstyle 
\langle \langle \delta\otimes \sigma, \varphi \rangle , \psi\rangle  
:=
\varphi(0) \langle \sigma,\psi\rangle \;
\text{ for all }\varphi \in \calD(\mR)\text{ and }\psi \in\calD(\mR^d).
$$ 
For $u\in \calL(\calD(\mR),\calD'(\mR^d))$, the derivative $\frac{d u}{dt}\in \calL(\calD(\mR),\calD'(\mR^d))$ is 
defined by 
$$
\textstyle 
\langle \langle u, \varphi \rangle , \psi \rangle 
:=
-\langle \langle u, \varphi'\rangle , \psi \rangle 
\;\text{ for all }\varphi \in \calD(\mR)\text{ and }\psi \in\calD(\mR^d).
$$ 
We now show a `jump rule' allowing differentiation of vector-valued functions in the sense of distributions. This is reminiscent of the jump rule for distributions of one variable; see e.g. \cite[Example~3.1.4, p.28]{Kun}.

\goodbreak 

\begin{proposition}[Jump Rule]
\label{prop_jump}
$\;$

\noindent Let $u\in C^1(\mR\setminus \{0\},\calD'(\mR^d))$  be such that 
$$
\textstyle 
\begin{array}{llll}
u(0+):=\lim\limits_{t\searrow 0} u(t),&\quad
u(0-):=\lim\limits_{t\nearrow 0} u(t),\;\\[0.39cm]
\frac{du}{dt} (0+):=\lim\limits_{t\searrow 0} \frac{du}{dt}(t),&\quad
\frac{du}{dt}(0-):=\lim\limits_{t\nearrow 0} \frac{du}{dt}(t)
\end{array}
$$
exist in $\calD'(\mR^d)$. Let $\sigma:=u(0+)-u(0-)\in \calD'(\mR^d)$. 

\smallskip 

\noindent Then $u, \frac{du}{dt}\in \text{\em L}^1_{\text{\em loc}}(\mR, \calD'(\mR^d))$, and 
 $\;
\textstyle 
\frac{d T_{u}}{dt\hfill }=T_{\frac{du}{dt}} +\delta\otimes \sigma.
$ 
\end{proposition}
\begin{proof} That $u, \frac{du}{dt}\in \text{L}^1_{\text{loc}}(\mR, \calD'(\mR^d))$ is clear from the definition. It remains to show that $\frac{d T_{u}}{dt\hfill }=T_{\frac{du}{dt}} +\delta\otimes \sigma$. We first show this in the scalar case, when $u$ takes complex values. Let $\varphi \in \calD(\mR)$, and suppose that $\varphi$
  is $0$ outside $[-a, a]$,  $a>0$. Then
$$
\textstyle 
\begin{array}{rcl}
\langle \frac{dT_u}{dt\hfill},\varphi \rangle 
\!\!\!&=&\!\!\!
-\langle T_u,\varphi' \rangle 
=
-\int_{-a}^a u(t)\varphi'(t) \;\!dt 
\\[0.15cm]
\!\!\!&=&\!\!\!-\int_{-a}^0 u(t)\varphi'(t)\;\! dt -\int_0^a u(t)\varphi'(t) \;\!dt 
\\[0.15cm]
\!\!\!&=&\!\!\! \int_{-a}^0 u'(t)\varphi(t) \;\!dt-u(0-)\varphi(0) 
\\[0.15cm]
\!\!\!&&\!\!\!+\int_0^a u'(t)\varphi(t) \;\!dt+u(0+)\varphi(0) \\[0.15cm]
\!\!\!&=&\!\!\! \int_{-a}^a u'(t)\varphi(t) \;\!dt+(u(0+)-u(0-))\varphi (0)\\[0.15cm]
\!\!\!&=&\!\!\! \langle T_{u'},\varphi\rangle+(u(0+)-u(0-))\langle \delta,\varphi
\rangle 
\\[0.15cm]
\!\!\!&=&\!\!\!\langle T_{u'}+ (u(0+)-u(0-))\delta ,\varphi\rangle.
\end{array}
$$
The general case follows from the scalar case applied to the function $t\mapsto \langle u(t), \psi\rangle$, $\psi\in \calD(\mR^d)$, noting that it has the derivative $\langle \frac{du}{dt}(t),\psi\rangle $ outside $0$, and has the jump $\langle \sigma, \psi\rangle$ at $0$.
\end{proof}

\section{The partial differential equation case}
\label{section_3}

\noindent Every polynomial $p\in \mC[x_1,\cdots,x_d][t]$ gives a natural differential operator $D_p:\calL(\calD(\mR),\calD'(\mR^d))\to 
\calL(\calD(\mR),\calD'(\mR^d))$ as follows: For a polynomial $p=a_0+a_1t+\cdots+a_n t^n$, where $a_0,\cdots, a_n\in 
 \mC[x_1,\cdots,x_d]$,  $D_p$ is given by 
$$
\textstyle 
D_p u 
= a_0(\frac{\partial\hfill}{\partial x_1}, \cdots, \frac{\partial\hfill}{\partial x_d}) u
+a_1(\frac{\partial\hfill}{\partial x_1}, \cdots, \frac{\partial\hfill}{\partial x_d}) \frac{du}{dt}
+\cdots+a_n(\frac{\partial\hfill}{\partial x_1}, \cdots, \frac{\partial\hfill}{\partial x_d}) \frac{d^nu}{dt^n},
$$
for all $u\in \calL(\calD(\mR),\calD'(\mR^d))$. 

For $p\in \mC[x_1,\cdots,x_d][t]$, let $S_p$ denote the set of continuously differentiable solutions, that is, 
$$
\textstyle
S_p=\{ u\in C(\mR, \calD'(\mR^d)): D_p u=0\text{ in the sense of distributions}\}.
$$
Here, $D_p u=0$ in the sense of distributions means $D_p T_u=0$, where $T_u$ denotes the regular distribution corresponding to $u\in \text{L}^1_{\text{loc}}(\mR, \calD'(\mR^d))$. 

If $u_1,u_2\in S_p\cap C^1(\mR,\calD'(\mR^d)))$ and $u_1(0)\!=\!u_2(0)\in \calD'(\mR^d)$, then we define their {\em concatenation} 
$u_1\oamp  u_2 \in C(\mR,\calD'(\mR^d))$ by 
$$
\textstyle 
(u_1\oamp u_2)(t)=\Big\{\! \begin{array}{l} 
u_1(t)  \;\text{ if }\; t\!\le\! 0,\\
u_2(t) \;\text{ if } \; t\! \ge\!  0.
\end{array}
$$
We say that the set $S_p$ has the {\em concatenability property} if whenever $u_1,u_2\in S_p\cap C^1(\mR,\calD'(\mR^d))$, their concatenation $u_1\oamp  u_2$ belongs to $S_p$. 

For a polynomial $p=a_0+a_1t+\cdots+a_n t^n\in 
  \mC[x_1,\cdots,x_d][t]$, where $n\in \mN\cup\{0\}$, $a_0,\cdots, a_n\in 
 \mC[x_1,\cdots,x_d]$,  and $a_n\neq 0$, we call $n$ the {\em $t$-degree of $p$}.

\begin{theorem}
\label{PDE_thm}
Let $p\in \mC[x_1,\cdots,x_d][t]$ have $t$-degree $n\ge 1$. 

\noindent Then $S_p$ has the concatenability property if and only if $n=1$.
\end{theorem}
\begin{proof} (If part:) Let $p\!=\!a_0+a_1t$, where $a_0,a_1\in \mC[x_1,\cdots, x_d]$ and $a_1\!\neq \!0$. 
Let $u_1, u_2\in S_p\cap C^1(\mR,\calD'(\mR^d)))$ be such that $u_1(0)\!=\!u_2(0)$. We denote $u_1 \oamp u_2$ by $u$ in order to avoid notational clutter. We have  
$$
\textstyle 
\begin{array}{llll}
u(0+)\!=\!\lim\limits_{t\searrow 0} u_2(t)\!=\!u_2(0),&\quad
u(0-)\!=\!\lim\limits_{t\nearrow 0} u_1(t)\!=\!u_1(0),\;\\[0.39cm]
\frac{du}{dt} (0+)\!=\!\lim\limits_{t\searrow 0} \frac{du_2}{dt\hfill}(t)\!=\!\frac{du_2}{dt\hfill}(0),&\quad
\frac{du}{dt}(0-)\!=\!\lim\limits_{t\nearrow 0} \frac{du_1}{dt\hfill}(t)\!=\!\frac{du_1}{dt\hfill}(0).
\end{array}
$$
By the jump rule (Proposition~\ref{prop_jump}), we have 
$$
\textstyle 
\frac{dT_{u}}{dt\hfill}
=
T_{\frac{du}{dt}} +\delta \otimes (u_2(0)-u_1(0))
=
T_{\frac{du}{dt}} +\delta \otimes 0
=
T_{\frac{du}{dt}}.
$$
Hence $D_p T_u= T_{D_p u}$. But $D_p u_1=0=D_p u_2\in C(\mR, \calD'(\mR^d))$. So for all $\varphi \in \calD(\mR)$ and $\psi \in \calD(\mR^d)$, we have 
$$
\textstyle 
\begin{array}{rcl}
\langle \langle D_p T_u, \varphi \rangle, \psi \rangle 
\!\!\!&=&\!\!\!
\langle \langle T_{D_pu}, \varphi \rangle, \psi \rangle
=
\int_\mR \varphi (t) \langle (D_pu)(t), \psi \rangle \;\!dt 
\\[0.21cm]
\!\!\!&=&\!\!\!
\int_{-\infty}^0 \varphi (t) \langle (D_pu_1)(t), \psi \rangle \;\!dt 
+
\int_0^{\infty} \varphi (t) \langle (D_pu_2)(t), \psi \rangle \;\!dt 
\\[0.21cm]
\!\!\!&=&\!\!\!
\int_{-\infty}^0 \varphi (t) \langle 0, \psi \rangle \;\!dt 
+
\int_0^{\infty} \varphi (t) \langle 0, \psi \rangle \;\!dt 
=0+0=0.
\end{array}
$$
Thus $u_1 \oamp u_2=u\in S_p$. Hence $S_p$ has the concatenability property. 

\vspace{0.21cm}

\noindent (Only if part:) Suppose that $S_p$ has the concatenability property, the integer $n$ is strictly bigger than $1$,  $p\!=\!a_0+a_1t+\cdots+a_n t^n\in 
  \mC[x_1,\cdots,x_d][t]$, where $a_0,\cdots, a_n\in 
 \mC[x_1,\cdots,x_d]$,  and $a_n\!\neq \!0$. Let $\xi \in \mR^d$ be such that $a_n(\xi)\neq 0$. 
 Let $p_{\xi}:=a_0(\xi)+a_1(\xi) t+\cdots +a_n(\xi)t^n\in \mC[t]$. Then $p_{\xi}$ has degree $n>1$. Let $\widetilde{u}_1,\widetilde{u}_2\in S_{p_\xi}\cap C^1(\mR)$ and  $\widetilde{u}_1(0)=\widetilde{u}_2(0)$. Then $p_{\xi}(\frac{d\hfill}{dt})\widetilde{u}_1=0$ and $p_{\xi}(\frac{d\hfill}{dt})\widetilde{u}_2=0$. Define $u_1,u_2\in C^1(\mR, \calD'(\mR^d))$ by $u_i(t)=\widetilde{u}_i(t) e^{\xi\cdot x} $, $i\in \{1,2\}$. Here $\xi\cdot x$ denotes the usual Euclidean inner product of $\xi$ and $x$ in $\mR^d$. Then $u_1,u_2\in S_p$ because 
 $$
 \textstyle 
 D_p u_i =p_{\xi}(\frac{d\hfill}{dt})\widetilde{u}_i e^{\xi\cdot x}=0e^{\xi\cdot x}=0, \quad i\in \{1,2\}.
 $$
 Also $u_1(0)=\widetilde{u}_1(0)e^{\xi\cdot x} =\widetilde{u}_2(0)e^{\xi\cdot x} =u_2(0)$. 
Since $S_p$ has the concatenability property, $u\!:=\!u_1 \oamp u_2\in S_p$. But $(u_1 \oamp u_2)(t)=(\widetilde{u}_1\oamp\widetilde{u}_2)(t)e^{\xi\cdot x}$. Denote $\widetilde{u}_1\oamp\widetilde{u}_2$ by $\widetilde{u}$. 
So we have 
$$
\textstyle 
0=D_p T_u = p_{\xi}(\frac{d\hfill}{dt}) T_{\widetilde{u}} e^{\xi\cdot x},
$$
giving $\widetilde{u}_1\oamp\widetilde{u}_2=\widetilde{u}\in S_{p_{\xi}}$. Thus $S_{p_{\xi}}$ has the concatenability property. This implies by Theorem~\ref{one_variable_thm} that the degree of $p_{\xi}$ must be one, a contradiction. Hence it is impossible that $n>1$, and so $n=1$. 
 \end{proof}

\begin{remark}
Instead of considering the spatial values to be distributions in $\calD'(\mR^d)$, one can also consider 
$S_p$ to comprise those $u$ whose spatial values lie in the smaller space $\calS'(\mR^d)$ of tempered distributions on $\mR^d$. Thus 
 $$
\textstyle
S_p=\{ u\in C(\mR, \calS'(\mR^d)): D_p u=0\text{ in the sense of distributions}\}.
$$ 
Then Theorem~\ref{PDE_thm} continues to be true. The proof is the same, mutatis mutandis, 
except that in the `only if' part, we set $u_i=\widetilde{u}_i e^{i \xi\cdot x} $ (i.e., the spatial factor $e^{i \xi\cdot x}$ bears an `$i$' in the exponential, to ensure that it is an element of $\calS'(\mR^d)$). 

A similar result is also true for spatial profiles belonging to periodic distributions (see \cite{Sas}) with analogous changes made in the proof. 
\end{remark}

\smallskip 

\noindent {\bf Funding.} No external funding was received. 

\smallskip

\noindent {\bf Data Availability Statement}. Data sharing is not applicable to this article as no new
data were created or analysed in this study.

\smallskip

\noindent {\bf Declarations
Conflict of interest}. The authors declare that they have no conflicts of interest to this
work.

\smallskip

\noindent {\bf Acknowledgements.} The authors thank the two anonymous referees for their careful reading of the article, and useful comments.


\begin{thebibliography}{99}
 
\bibitem{Car}
R. Carroll. 
{\em Abstract Methods in Partial Differential Equations}. 
Harper's Series in Modern Mathematics, Harper and Row, 1969.
  
\bibitem{CotSas}
A. Sasane and T. Cotroneo. 
Conditions for time-controllability of behaviours. 
{\em International Journal of Control}, 75:61-67, no. 1, 2002.
    
 \bibitem{Hor0}
 L. H\"{o}rmander. 
 On the division of distributions by polynomials. 
   {\em Arkiv f\"{o}r Matematik}, 3:555-568, 1958.
    
 \bibitem{Hor}
L. H\"{o}rmander. 
Null solutions of partial differential equations. 
{\em Archive for Rational Mechanics and Analysis}, 4:255-261, 1960.
     
          
\bibitem{Kun}
M. Kunzinger. 
{\em Theory of Distributions}. Lecture notes, Universit\"at Wien, 2019.
     
 \bibitem{Loj}
 S. \L ojasiewicz. 
 Sur le probl\`eme de la division. 
 {\em Studia Mathematica}, 18:87-136, 1959.
    
    \bibitem{NapRapRoc}
    D. Napp, P. Rapisarda, and P. Rocha. 
   Time-relevant stability of $2$D systems. 
   {\em Automatica}, 47:2373-2382, no. 11, 2011.

 \bibitem{ObeSch}
 U. Oberst and M. Scheicher. 
 Time-autonomy and time-controllability of discrete
              multidimensional behaviours. 
              {\em International Journal of Control},
    85:990-1009, no. 8, 2012.
         
    
    
\bibitem{RapWil}
P. Rapisarda and J. Willems. 
State maps for linear systems. 
{\em SIAM Journal on Control and Optimization},
   35:1053-1091, no. 3, 1997.
   
 \bibitem{RocWil}
P. Rocha and J. Willems. 
Markovian properties for $2$D behavioral systems described by
              PDE's: the scalar case. 
              {\em Multidimensional Systems and Signal Processing},
    22:45-53, no. 1-3, 
      2011.   
        
        \bibitem{Sas00}
        A. Sasane. 
        On the Willems closure with respect to ${\mathscr{W}}_{\mathfrak{s}}$. 
   {\em IMA Journal of Mathematical Control and Information},
   20:217-232, no. 2, 2003.
      
        \bibitem{Sas0}
        A. Sasane. 
        Time-autonomy and time-controllability of $2$-D behaviours that are tempered in the spatial direction. {\em Multidimensional Systems and Signal Processing},15:97-116, 2004.
        
 \bibitem{Sas}
 A.  Sasane. 
 On the existence of spatially tempered null solutions to
              linear constant coefficient PDEs. 
 {\em Israel Journal of Mathematics}, 244:273-291, no. 1, 2021.
   
   \bibitem{SasThoWil}
   A. Sasane, E. Thomas, and J. Willems. 
     Time-autonomy versus time-controllability. 
      {\em Systems and Control Letters}, 45:145-153, no. 2, 2002.
    
     
\bibitem{SasWag}
A. Sasane and P. Wagner. 
Division problem for spatially periodic distributions. 
{\em Journal of Mathematical Analysis and Applications}, 
408:70-75, no. 1, 2013.
 
 \bibitem{vanRap}
 A.  van der Schaft and P. Rapisarda. 
 State maps from integration by parts. 
 {\em SIAM Journal on Control and Optimization},
   49:2415-2439, no. 6, 2011.
      
          
\bibitem{Wil}
J. Willems. 
State and first order representations. 
   In {\em Unsolved Problems in Mathematical Systems and Control Theory}, V. Blondel, A. Megretski (editors), Princeton University Press, 2004.
    
\bibitem{Zer}
E. Zerz. 
First-order representations of discrete linear
              multidimensional systems. 
              {\em Multidimensional Systems and Signal Processing}, 11:359-380, no. 4, 2000.
              
    \vspace{0.45cm}
   
\end{thebibliography}
\end{document}